\documentclass[10pt]{amsart}

\usepackage{hyperref}
 \hypersetup{ colorlinks=true, linkcolor=blue, filecolor=magenta, urlcolor=cyan, }

\newcommand*{\Scale}[2][4]{\scalebox{#1}{$#2$}}%

\usepackage{hyperref}
 \hypersetup{ colorlinks=true, linkcolor=blue, filecolor=magenta, urlcolor=cyan, }

\usepackage{graphics}
\usepackage{graphicx}
\usepackage{epsfig}
\usepackage{amsmath}
\usepackage{graphics}
\usepackage{graphicx}
\usepackage{epstopdf}
\usepackage{multirow}
\usepackage{booktabs}
\usepackage{ifpdf}
\usepackage{color}
\usepackage{cite}
\usepackage{amssymb}
\usepackage{amsthm}
\usepackage{amsmath}
\usepackage{matlab-prettifier}

\theoremstyle{plain}
\newtheorem{ntn}{Notations}[section]

\newtheorem{dfn}{Definition}[section]
\newtheorem{thm}{Theorem}[section]
\newtheorem{cor}{Corollary}[section]

\usepackage{lineno}
\numberwithin{equation}{section}

\begin{document}

\bibliographystyle{plain}

\title{A New Formula of the Determinant Tensor with Symmetries}

\author{Jeong-Hoon Ju}
\email{jjh793012@naver.com}

\author{Taehyeong Kim}
\email{th\_kim@pusan.ac.kr}

\author{Yeongrak Kim}
\email{yeongrak.kim@pusan.ac.kr}

\address{Department of Mathematics, Pusan National University, 
2 Busandaehak-ro 63beon-gil, Geumjeung-gu, 46241 Busan, Republic of Korea}

\begin{abstract}
In this paper, we present a new formula of the determinant tensor $det_n$ for $n \times n$ matrices. In \cite{kim2023newdet4} , Kim, Ju, and Kim found a new formula of $4 \times 4$ determinant tensor $det_4$ which is available when the base field is not of characteristic $2$. Considering some symmetries in that formula, we found a new formula so that
\begin{equation*}
	\operatorname{Crank}(det_n) \leq \operatorname{rank}(det_n) \leq \frac{n!}{2^{\lfloor(n-2)/2 \rfloor}}
\end{equation*}
when the base field is not of characteristic $2$.
\vspace{0.5cm}
\end{abstract}



\maketitle


\section{Introduction}

Let $\mathbb{K}$ be a field, let $V$ be a $\mathbb{K}$-vector space of dimension $n$, and let $\{e_1,...,e_n\}$ be a basis for the dual vector space $V^*$. The $n \times n$ determinant tensor $det_n \in (V^*)^{\otimes n}$ is defined as
\begin{equation}\label{Detn}
	det_n=\sum_{\sigma \in S_n} sgn(\sigma) ~e_{\sigma(1)} \otimes e_{\sigma(2)} \otimes \cdots \otimes e_{\sigma(n)},
\end{equation}
where $S_n$ is the symmetric group of $n$ letters and $sgn(\sigma)$ denotes the sign of the permutation $\sigma \in S_n$, as in \cite{derksen2016nuclear, krishna2018tensor, houston2023new}. Alternatively, the determinant of an $n \times n$ matrix $A$ is defined as a homogeneous polynomial of degree $d$ on the entries of $A$
\begin{equation*}
	\det(A) = \sum_{\sigma \in S_n}  \operatorname{sgn}(\sigma) \prod_{i=1}^n A_{i, \sigma(i)}.
\end{equation*}
By considering the space of $n \times n$ matrices as the space of $n$-tuples of row (column) vectors in $V$, the determinant can be regarded as an $n$-linear function with respect to columns of the matrices. This is why we can represent the determinant as a tensor $det_n$ as in (\ref{Detn}). 

The above argument identifies the determinant function as a tensor of order $n$, we consider its tensor rank. 
Roughly speaking, the tensor rank is the smallest number of decomposable tensors to represent the given tensor as a sum of decomposable tensors. The tensor rank measures the complexity of a given tensor, or equivalently, a multilinear map. It is natural to ask the tensor rank of the determinant tensor, as well as the matrix multiplication tensor or the permanent tensor, as we can find in various works including \cite{derksen2019explicit, conner2019new}.

There are two major viewpoints to see the determinant as tensors: either as a tensor living in $(\mathbb{K}^n)^{\otimes n}$, equivalently, an $n$-linear alternating function, or as a homogeneous polynomial of degree $n$ in $n^2$ variables which is the determinant of the $n \times n$ matrix $(x_{i,j})$ of independent variables $x_{1,1}, \cdots, x_{n,n}$. For the first convention, a standard measurement is the tensor rank (denoted by $rank$), and for the second convention, both the Waring rank and the Chow rank (denoted by $Wrank$ and $Crank$, respectively) are used. We will not carefully deal with the Waring rank of the determinant in this paper, and we refer to \cite{johns2022improved} for the very recent developments in the study of Waring rank of the determinant. In any case, these rank notions encode a huge amount of information about algebro-geometric meaning of the given tensor together with the development on geometric complexity theory, for instance, the notion of border rank and secant spaces of classical algebraic varieties, see \cite{landsberg2012tensors, bernardi2018hitchhiker}. For order $2$ tensors $\mathcal T \in V_1 \otimes V_2$, the tensor rank of $\mathcal T$ coincides with the rank of a corresponding matrix which is well understood. Unlikely the rank of matrices, the tensor rank of a tensor of order $d \ge 3$ is mysterious, and it is difficult to find good upper and lower bounds of the rank of a given tensor in many cases. The explicit tensor rank and the Waring rank of $\det_n$ are widely unknown for $n \ge 4$.

Let us briefly review previous studies on the rank of $\det_n$. In \cite{derksen2016nuclear}, Derksen proved that $\operatorname{rank}({\det}_3) \leq 5$ when $\operatorname{char}(\mathbb{K}) \neq 2$, by exhibiting an explicit formula 
\begin{equation}\label{Derk}
	\begin{aligned}
		{\det}_3 =\frac{1}{2}(& (e_3+e_2)\otimes (e_1-e_2) \otimes (e_1+e_2)\\
				&+(e_1+e_2)\otimes (e_2-e_3) \otimes (e_2+e_3)\\
				&+2e_2 \otimes (e_3-e_1) \otimes (e_3+e_1) \\
				&+(e_3-e_2)\otimes (e_2+e_1) \otimes (e_2-e_1) \\
				&+(e_1-e_2) \otimes (e_3+e_2) \otimes (e_3-e_2))
	\end{aligned}
\end{equation}
consisted of $5$ decomposable tensors. Using this formula and the generalized Laplace expansion, Derksen also showed that $\operatorname{rank}({\det}_n)\leq \left(5/6 \right)^{\lfloor n/3 \rfloor}\cdot n!$ if $\operatorname{char}(\mathbb{K}) \neq 2$. Furthermore, Derksen proved that $\operatorname{rank}({\det}_n) \geq {n \choose \lfloor n/2 \rfloor}$ using the flattenings. 
In \cite{derksen2019explicit}, when $\operatorname{char}(\mathbb{K}) \neq 2$, Derksen and Makam proved that $\operatorname{rank}({\det}_3)=5$ by showing that the border rank (denoted by $\operatorname{brank}$) of ${\det}_3$, which is smaller or equal to the tensor rank, is $5$. In \cite{krishna2018tensor}, Krishna and Makam gave a new formula of ${\det}_3$ without using the coefficient $\frac{1}{2}$ so that $\operatorname{rank}({\det}_n)\leq \left(5/6 \right)^{\lfloor n/3 \rfloor}\cdot n!$ is valid for arbitary field $\mathbb{K}$. They used the  Koszul flattening to show that $\operatorname{rank}({\det}_3)=5$ over arbitrary field $\mathbb{K}$, and that $17 \leq \operatorname{brank}({\det}_5) \leq \operatorname{rank}({\det}_5)$ and $62 \leq \operatorname{brank}({\det}_7) \leq \operatorname{rank}({\det}_7)$ over arbitrary $\mathbb{K}$. In the case of $det_4$, a result in \cite{derksen2015lower} and the well-known inequality $\operatorname{Wrank}(det_n) \leq 2^{n-1}\cdot \operatorname{rank}(det_n)$ tell us that $7 \leq \operatorname{rank}(det_4)$ holds over an algebraically closed field. Very recently, Houston, Goucher, and Johnston reported a significant improvement on upper bounds of $\operatorname{rank}({\det}_n)$ by the $n$-th Bell's number for arbitrary $\mathbb{K}$ using combinatorics \cite{houston2023new}. In particular, they provided an explicit formula for $\det_4$ consisting of $15$ terms so that $\operatorname{rank}({\det}_4) \leq 15$. Their formula also works over fields of positive characteristics, in particular, their formula is consisted of $12$ terms only when $\operatorname{char}(\mathbb{K})=2$. They also gave a computer-assisted proof that $\operatorname{rank}({\det}_4) = 12$ when $\operatorname{char}(\mathbb{K})=2$.
 
The authors of the paper were inspired by Derksen's formula (\ref{Derk}), and found a new formula of the $4 \times 4$ determinant tensor by using the Least Absolute Shrinkage and Selection Operator (LASSO) \cite{kim2023newdet4}:
\begin{equation}\label{newdet4}
\begin{aligned}
	{\det}_4= \frac{1}{2}(&(e_1-e_2)\otimes (e_3-e_4)\otimes (e_3+e_4)\otimes (e_1+e_2)\\
		&-(e_1-e_3)\otimes (e_2-e_4)\otimes (e_2+e_4)\otimes (e_1+e_3)\\
		&+(e_1-e_4)\otimes (e_2-e_3)\otimes (e_2+e_3)\otimes (e_1+e_4)\\
		&+(e_2-e_3)\otimes (e_1-e_4)\otimes (e_1+e_4)\otimes (e_2+e_3)\\
		&-(e_2-e_4)\otimes (e_1-e_3)\otimes (e_1+e_3)\otimes (e_2+e_4)\\
		&+(e_3-e_4)\otimes (e_1-e_2)\otimes (e_1+e_2)\otimes (e_3+e_4)\\
		&+(e_1+e_2)\otimes (e_3+e_4)\otimes (e_3-e_4)\otimes (e_1-e_2)\\
		&-(e_1+e_3)\otimes (e_2+e_4)\otimes (e_2-e_4)\otimes (e_1-e_3)\\
		&+(e_1+e_4)\otimes (e_2+e_3)\otimes (e_2-e_3)\otimes (e_1-e_4)\\
		&+(e_2+e_3)\otimes (e_1+e_4)\otimes (e_1-e_4)\otimes (e_2-e_3)\\
		&-(e_2+e_4)\otimes (e_1+e_3)\otimes (e_1-e_3)\otimes (e_2-e_4)\\
		&+(e_3+e_4)\otimes (e_1+e_2)\otimes (e_1-e_2)\otimes (e_3-e_4)).
\end{aligned}
\end{equation}
which is valid when $\operatorname{char}(\mathbb{K}) \neq 2$. Here, the $12$ summands enjoy a number of symmetries, so it sounds promising that this formula can be extended in a similar way for larger matrices. Indeed, we generalized the formula (\ref{newdet4}) for square matrices of even size $n=2k$ as follows.

\begin{thm}\label{MainThm}
	Let $n=2k$ for some positive integer $k$, and denote $[n]=\{1,2,...,n\}$. Let $I_{2k}$ be the index set
	\begin{equation*}
	\Scale[0.9]{
		I_{2k}=\{ (i_1,j_1, \cdots, i_k,j_k) \in \mathbb{Z}^{2k}~|~ i_1,j_1, \cdots ,i_k,j_k \in [2k] \text{ are pairwise different, and } i_{p}<j_{p} \text{ for each } 1 \leq p \leq k\},}
	\end{equation*}
	and let 
	\begin{equation*}
		S_{a_1, a_2, \cdots ,a_{2k}}=sgn\begin{pmatrix}
		1 & 2 & \cdots  & 2k\\
		a_1 & a_2 & \cdots  & a_{2k}
	\end{pmatrix}.
	\end{equation*}	
	Then
	\begin{equation}\label{GeneralFormula}
	\Scale[0.9]{
		\begin{aligned}
			det_{n}&=\frac{1}{2}\sum_{I_{2k}}S_{i_1,j_1,\cdots,i_k,j_k}(e_{i_1}-e_{j_1})\otimes (e_{i_2}-e_{j_2})\otimes \cdots \otimes (e_{i_k}-e_{j_k})\otimes (e_{i_k}+e_{j_k}) \otimes \cdots \otimes (e_{i_2}+e_{j_2})\otimes (e_{i_1}+e_{j_1})\\
		&+\frac{1}{2}\sum_{I_{2k}}(-1)^kS_{i_1,j_1,\cdots,i_k,j_k}(e_{i_1}+e_{j_1})\otimes (e_{i_2}+e_{j_2})\otimes \cdots \otimes (e_{i_k}+e_{j_k})\otimes (e_{i_k}-e_{j_k}) \otimes \cdots \otimes (e_{i_2}-e_{j_2})\otimes (e_{i_1}-e_{j_1}).
		\end{aligned}
		}
	\end{equation}
\end{thm}
Note that the sign of the permutation $S_{a_1, \cdots, a_{2k}}$ coincides with the determinant of the corresponding permutation matrix. Consequently, we obtain 
\begin{equation*}
	\operatorname{Crank}(det_n) \leq \operatorname{rank}(det_n) \leq \frac{n!}{2^{\lfloor(n-2)/2 \rfloor}}
\end{equation*}
which greatly improves Derksen's upper bound.

The structure of the paper is as follows. In Section \ref{Sect:Preliminaries}, we review basic notions, in particular, three notions of ranks of tensors, namely, the tensor rank, the Waring rank, and the Chow rank. In Section \ref{Sect:Symmetries}, we analyze the symmetries and the sign convention of the summands appearing in the formula (\ref{newdet4}) and derive a new formula Theorem \ref{MainThm}. Using this new formula, we obtain an upper bound for the tensor rank of $\det_n$ and compare it with recent studies.

\section*{Acknowledgement}
J.-H. J. and Y. K. are supported by the Basic Science Program of the NRF of Korea [NRF-2022R1C1C1010052]. J.-H. J. participated the introductory school of AGATES in Warsaw (Poland) and thanks the organizers for providing a good research environment throughout the school. The authors thank Hyun-Min Kim for his invaluable advice, encouragement, and helpful discussions.


\section{Preliminaries}\label{Sect:Preliminaries}

We briefly review various notions for the rank of a tensor. We follow the definition of the determinant tensor in \cite{derksen2016nuclear,krishna2018tensor,houston2023new}, and mostly follow the definitions and conventions in \cite{landsberg2012tensors, bernardi2018hitchhiker}.

\begin{ntn}
Throughout this paper, we use the following notations:
\begin{itemize}
\item $\mathbb{K}$ : a field of characteristic $\neq 2$;
\item $V, V_i, W$ : finite dimensional $\mathbb{K}$-vector spaces;
\item $V^*$ : the dual vector space of $V$;
\item $[d] = \{1, 2, \cdots, d \}$ where $d$ is a positive integer.
\end{itemize}
\end{ntn}

\begin{dfn}[Multilinear map and tensor]\label{DefTensor}
	A map $\varphi:V_1 \times V_2 \times \cdots \times V_d \rightarrow W$ is said to be \emph{multilinear} if it is linear with respect to each vector space $V_i$ for $i \in [d]$. The space of such multilinear maps is identical to $V_1^* \otimes V_2^* \otimes \cdots \otimes V_d^* \otimes W$. An element $\mathcal{T} \in V_1^* \otimes V_2^* \otimes \cdots \otimes V_d^* \otimes W$ is called a \emph{tensor}, and the number of vector spaces appearing in this tensor product is called the \emph{order of $\mathcal{T}$}.
\end{dfn}

\begin{dfn}[Determinant tensor]
	Let $V$ be a $\mathbb{K}$-vector space of dimension $n$. The Cartesian $n$-product $\underbrace{V \times \cdots \times V}_{n\text{ copies}}$ is identical to the space of $n \times n$ square matrices $Mat_{n}(\mathbb{K})$ with entries in $\mathbb{K}$. We consider $\det$ as an $n$-linear function from $\underbrace{V \times \cdots \times V}_{n\text{ copies}}$ to $\mathbb{K}$. Tensoring by $\mathbb{K}$ does not change the tensor product, so we may regard it as a tensor of order $n$ (not of order $n+1$) in $\underbrace{V^* \otimes \cdots \otimes V^*}_{n \text{ copies}}$. Indeed,
	\begin{equation*}
		{\det}_n= \sum_{\sigma \in S_n} ~ sgn(\sigma)e_{\sigma(1)} \otimes e_{\sigma(2)} \otimes \cdots \otimes e_{\sigma(n)}
	\end{equation*}
where $\{e_1,...,e_n\}$ is a basis of $V^*$. For instance, (\ref{Detn}) gives $det_2=e_1 \otimes e_2 - e_2 \otimes e_1$. We have
	\begin{equation*}
		det_2\left(\begin{bmatrix}
			a_{1,1} & a_{1,2}\\ a_{2,1} & a_{2,2}
		\end{bmatrix}
		\right)=e_1\left(\begin{bmatrix}
			a_{1,1} \\ a_{2,1}
		\end{bmatrix}
		\right)e_2\left(\begin{bmatrix}
			a_{1,2} \\ a_{2,2}
		\end{bmatrix}
		\right)-e_2\left(\begin{bmatrix}
			a_{1,1} \\ a_{2,1}
		\end{bmatrix}
		\right)e_1\left(\begin{bmatrix}
			a_{1,2} \\ a_{2,2}
		\end{bmatrix}
		\right)=a_{1,1}a_{2,2}-a_{2,1}a_{1,2}
	\end{equation*}
which coincides with the usual convention. 
\end{dfn}

Since $\dim V$ is finite, the dual vector space $V^*$ is (non-canonically) isomorphic to $V$, and hence we do not really have to distinguish a vector space and its dual space when we only focus on the tensor rank and similar notions. 

\begin{dfn}[Tensor rank]
	Let $V_1,V_2,...,V_d$ be $\mathbb{K}$-vector spaces, and let $\mathcal{T} \in V_1 \otimes V_2 \otimes \cdots \otimes V_d$ be a tensor of order $d$. Then 
	\begin{equation*}
		r=\min\left\{ k ~\middle|~\mathcal{T}=\sum_{i=1}^k v_{i,1} \otimes v_{i,2} \otimes \cdots \otimes v_{i,d} \text{ where } v_{i,j}\in V_{j} \text{ for each } j \in [d] \right\}
	\end{equation*}
	is called the \emph{(tensor) rank of $\mathcal{T}$}, and denoted by $\operatorname{rank}(\mathcal{T})$. A tensor of rank $\le 1$ is said to be \emph{decomposable}.
\end{dfn}

It is clear that the rank is invariant under the change of bases, so we do not have to worry about a choice of bases. As we discussed above, $\det_n$ can be also seen as a homogeneous polynomial of degree $n$ in $n^2$ variables. There are two commonly used rank notions to measure the complexity of homogeneous polynomials, namely, the Waring rank and the Chow rank of symmetric tensors. 

\begin{dfn}[Symmetric tensor]
	Let $V$ be a vector space, and let $\mathcal{T} \in V^{\otimes d}$. 
	If $\mathcal{T} \circ \sigma=\mathcal{T}$ for every permutation $\sigma \in S_d$, then $\mathcal{T}$ is called a \emph{symmetric tensor}.
\end{dfn}

A symmetric tensor is often represented as a homogeneous polynomial (see \cite{landsberg2012tensors, bernardi2018hitchhiker}). For example, the symmetric tensor $x \otimes y \otimes z+x \otimes z \otimes y + y \otimes x \otimes z + y \otimes z \otimes x+ z \otimes x \otimes y+ z \otimes y \otimes x$ is represented by the homogeneous polynomial $6xyz$ if $\operatorname{char}(\mathbb{K}) =0$ or $> 3$. 
This observation suggests a formal description of the determinant tensor as a homogeneous polynomial in the following way, even though the determinant tensor is not a symmetric tensor: ${\det}_2=e_1 \otimes e_2-e_2 \otimes e_1 \neq e_2 \otimes e_1 - e_1 \otimes e_2$ unless $\operatorname{char}(\mathbb{K})=2$. Instead, we also consider the index of each component, which leads to an identification of $\det_n$ as a homogeneous polynomial of degree $n$ in $n^2$ independent variables. For instance, $\det_3$ can be regarded as the homogeneous polynomial $x_{1,1}x_{2,2}x_{3,3}-x_{1,1}x_{2,3}x_{3,2}-x_{1,2}x_{2,1}x_{3,3}+x_{1,2}x_{2,3}x_{3,1}+x_{1,3}x_{2,1}x_{3,2}-x_{1,3}x_{2,2}x_{3,1}$ by considering the (generic) matrix of indeterminates 
\begin{equation*}
	\begin{bmatrix} x_{1,1} & x_{1,2} & x_{1,3} \\ x_{2,1} & x_{2,2} & x_{2,3} \\ x_{3,1} & x_{3,2} & x_{3,3} \end{bmatrix}.
\end{equation*} 
For those symmetric tensors (equivalently, for homogeneous polynomials), both the Waring rank and the Chow rank are frequently considered to measure the complexity. Let $Sym(V) \simeq \mathbb{K}[x_1,...,x_n]$ denote the polynomial ring over $\mathbb{K}$ in variables $x_1,...,x_n$, and let $Sym^d (V) \simeq \mathbb{K}[x_1,...,x_n]_d$ denote the subspace of homogeneous polynomials of degree $d$.

\begin{dfn}[Waring rank]
	For $0 \neq f \in \mathbb{K}[x_1,...,x_n]_d$, the number
	\begin{equation*}
		r=\min\left\{k~\middle|~f=\sum_{i=1}^k c_i (l_i)^d,~\text{ where }~l_i \in \mathbb{K}[x_1,...,x_n]_1 , ~c_i \in \mathbb{K} \text{ for each }i \in [k]\right\}
	\end{equation*}
	is called the \emph{Waring rank} (or \emph{symmetric rank}) of $f$, and denoted by $\operatorname{Wrank}(f)$.
\end{dfn}

\begin{dfn}[Chow rank]
	For $0 \neq f \in \mathbb{K}[x_1,...,x_n]_d$, the number
	\begin{equation*}
		r=\min\left\{k~\middle|~f=\sum_{i=1}^k l_{i,1}l_{i,2}\cdots l_{i,d} ~\text{ where }~l_{i,j} \in \mathbb{K}[x_1,...,x_n]_1 \text{ for each } i \in [k] \text{ and } j \in [d]\right\}
	\end{equation*}
	is called the \emph{Chow rank} (or the \emph{product rank}) of $f$, and denoted by $\operatorname{Crank}(f)$.
\end{dfn}

Note that the above procedure can be made for an arbitrary tensor $\mathcal{T} \in V_1 \otimes \cdots \otimes V_d$ which does not have to be symmetric. Indeed, we can associate $\mathcal T$ as a homogeneous polynomial of degree $d$ in $\dim (V_1) \cdots \dim (V_d)$ variables (we still denote by $\mathcal T$ for this homogeneous polynomial) and consider its Waring rank or Chow rank. The following relations between the tensor rank of $\mathcal T$ (as a tensor of order $d$) and the Waring/Chow rank of $\mathcal T$ (as a homogeneous polynomial of degree $d$) are well understood in the case when $\operatorname{char}(\mathbb{K})=0$ or $\operatorname{char}(\mathbb{K}) > d$, see \cite{houston2023new, ilten2016product} for more details. 

\begin{equation}\label{ChowIneq}
	\operatorname{Crank}(\mathcal{T}) \leq \operatorname{rank}(\mathcal{T})
\end{equation}
and
\begin{equation}\label{WaringIneq}	
\operatorname{Wrank}(\mathcal{T}) \leq 2^{d-1}\cdot  \operatorname{rank}(\mathcal{T})
\end{equation}
We give small remarks on these inequalities. If a given tensor $\mathcal{T} \in V_1 \otimes \cdots \otimes V_d$ admits a decomposition 
\begin{equation}\label{TensorDecomp}
	\mathcal{T}=\sum_{i=1}^r v_{i,1} \otimes v_{i,2} \otimes \cdots \otimes v_{i,d}
\end{equation}
as a sum of decomposable tensors, then the associated homogeneous polynomial $\mathcal{T}$ also has a decomposition of the form
\begin{equation}\label{ChowDecomp}
	\mathcal{T}=\sum_{i=1}^r l_{i,1}l_{i,2}\cdots l_{i,d},
\end{equation} 
where each linear form $l_{i,j}$ is derived from the vector $v_{i,j}$ as a linear combination of basis for $V_j$. To be precise, when the vector $v_{i,j}$ can be written as $v_{i,j}=\sum_{k=1}^{m_j} a_{i,k} e_{j,k}$ where $m_j = \dim V_j$ and $\{e_{j,1}, \cdots, e_{j,m_j} \}$ is a basis for $V_j$, then we associate a linear polynomial $l_{i,j} = \sum_{k=1}^{m_j} a_{i,k} x_{j,k}$ in the place of $v_{i,j}$. In particular, the first inequality (\ref{ChowIneq}) is still valid when $0<\operatorname{char}(\mathbb{K}) \le d$.
The inequality (\ref{WaringIneq}) follows from a similar argument as in the above and a result on the  Waring rank of monomials \cite{ranestad2011rank}.

\section{Generalizing the formula using symmetries}\label{Sect:Symmetries}

Our main theorem is strongly inspired by an amount of symmetries on the formula $(\ref{newdet4})$. It is worthwhile to analyze these symmetries and the formula before to proceed. First of all, we forget about the coefficient $\pm \frac{1}{2}$ and concentrate on the indices composing each of the summands. Each term appearing in the formula is consisted of the tensor product of $4$ vectors, so let us say 
	\begin{equation*}
		t_1 \otimes t_2 \otimes t_3 \otimes t_4.
	\end{equation*}
Let $a,b,c,d$ be indices which satisfy the following conditions
	\begin{equation*}
		a,b,c,d \in [4]~\text{ are pairwise distinct, }~a<b~\text{ and }~ c<d.
	\end{equation*}
We can easily observe the following rules on the indices of the components.
	\begin{itemize}
		\item [(\romannumeral1)] If the first component is of the form $t_1=e_a-e_b$, then the second component is of the form $t_2=e_c-e_d$. On the other hand, if the first component is of the form $t_1=e_a+e_b$, then the second component is of the form $t_2=e_c+e_d$.
		\item [(\romannumeral2)] If the first two components are $t_1=e_a-e_b$ and $t_2=e_c-e_d$ (resp. $t_1=e_a+e_b$ and $t_2=e_c+e_d$), then the last two components are $t_3 = e_c+e_d$ and $t_4 = e_a+e_b$ (resp. $t_3 = e_c-e_d$ and $t_4 = e_a-e_b$). Hence, each term has the form
				\begin{equation}\label{form1}
					(e_a-e_b)\otimes (e_c-e_d) \otimes (e_c+e_d)\otimes (e_a+e_b)
				\end{equation}
				or
				\begin{equation}\label{form2}
					(e_a+e_b)\otimes (e_c+e_d) \otimes (e_c-e_d)\otimes (e_a-e_b).
				\end{equation}
		\item [(\romannumeral3)] If the term 
				\begin{equation*}
					(e_a-e_b)\otimes (e_c-e_d) \otimes (e_c+e_d)\otimes (e_a+e_b)
				\end{equation*} 
				appears in the formula, then the term 
				\begin{equation*}
					(e_a+e_b)\otimes (e_c+e_d) \otimes (e_c-e_d)\otimes (e_a-e_b)
				\end{equation*}
				also appears in the formula, and vice versa.
	\end{itemize}
	One can easily check that the formula (\ref{newdet4}) is consisted of all the possible terms (\ref{form1}) and (\ref{form2}) satisfying the conditions above.

	Let us analyze how the signs are determined. We check that the sign for each term of the form (\ref{form1}) is given by the sign of permutation
	\begin{equation*}
		sgn\begin{pmatrix}
  	  1 & 2 & 3 & 4\\
		a & b & c & d
  \end{pmatrix}.
  	\end{equation*}
	 For example, the second term 
	\begin{equation*}
		-\frac{1}{2}(e_1-e_3)\otimes (e_2-e_4)\otimes (e_2+e_4)\otimes (e_1+e_3)
	\end{equation*}
	in the formula (\ref{newdet4}) has the negative sign (= coefficient $-\frac{1}{2}$), and the sign of permutation $sgn\begin{pmatrix}
    1 & 2 & 3 & 4\\
		1 & 3 & 2 & 4
  \end{pmatrix} = -1$ is negative. This rule is quite natural since
  \begin{equation*}
  	sgn\begin{pmatrix}
    1 & 2 & 3 & 4\\
		a & b & c & d
  \end{pmatrix}=sgn\begin{pmatrix}
    1 & 2 & 3 & 4\\
		a & c & d & b
  \end{pmatrix},
  \end{equation*}
  and the right-hand side is exactly the coefficient of the term $e_a \otimes e_c \otimes e_d \otimes e_b$ in the original formula (\ref{Detn}).  
  
  Next, let us consider the sign for the term of the form (\ref{form2}). Recall that we took 
  \begin{equation}\label{TargetToCandi}
  	\dfrac{1}{2} ~ sgn\begin{pmatrix}
    1 & 2 & 3 & 4\\
		a & c & d & b
  \end{pmatrix}~e_a \otimes e_c \otimes e_d \otimes e_b
  \end{equation} 
  from a term in the form (\ref{form1}), and hence, to make it compatible with $\det_4$, we need an extra term same as (\ref{TargetToCandi}) which can only occur from the corresponding term in the form (\ref{form2}) from the rule (iii). Hence, the unique candidate among the terms in (\ref{GeneralFormula}) to have (\ref{TargetToCandi}) is 
  \begin{equation*}
  	\pm~\frac{1}{2}(e_a+e_b)\otimes (e_c+e_d) \otimes (e_c-e_d)\otimes (e_a-e_b),
  \end{equation*}
  and the sign must be
  \begin{equation*}
  	(-1)^2\cdot sgn\begin{pmatrix}
    1 & 2 & 3 & 4\\
		a & c & d & b
  \end{pmatrix}=(-1)^2 \cdot sgn\begin{pmatrix}
    1 & 2 & 3 & 4\\
		a & b & c & d
  \end{pmatrix}.
  \end{equation*}

 We need a correction-term by a power $(-1)^2$, since $e_d$ and $e_b$ appear in the later-half components with the negative sign. This sign convention is also necessary to show that the terms that we do not want indeed vanish. For example, we expand the second term of (\ref{newdet4})
  \begin{equation}\label{FirstTerm}
  	-\frac{1}{2}(e_1-e_3)\otimes (e_2-e_4)\otimes (e_2+e_4)\otimes (e_1+e_3),
  \end{equation}
we see that there is a bad term 
  \begin{equation}\label{vanish1}
  	-\frac{1}{2}e_1 \otimes e_2 \otimes e_4 \otimes e_1
  \end{equation}
  which does not contribute to the determinant tensor $\det_4$. This term is cancelled from the corresponding term (follows from rule (iii) again)
  \begin{equation*}
  	-\frac{1}{2}(e_1+e_3)\otimes (e_2+e_4)\otimes (e_2-e_4)\otimes (e_1-e_3),
  \end{equation*}
  whose expansion contains the term
  \begin{equation}\label{vanish2}
  	\frac{1}{2}e_1 \otimes e_2 \otimes e_4 \otimes e_1
  \end{equation}
  so that (\ref{vanish1})$+$(\ref{vanish2})$=0$. 
  
  There are further bad terms which may not be killed by this procedure. For instance, the decomposable tensor (\ref{FirstTerm}) also yields a bad term
  \begin{equation}\label{vanish3}
  		-\frac{1}{2}e_1 \otimes e_2 \otimes e_2 \otimes e_1
  \end{equation}
   which we do not want. To kill it, we need to look for a summand where the places of $e_3$ and $e_4$ are switched, in particular,    
   the third term of (\ref{newdet4})
   \begin{equation*}
   	\frac{1}{2}(e_1-e_4)\otimes (e_2-e_3)\otimes (e_2+e_3)\otimes (e_1+e_4)
   \end{equation*}
   contains the term
   \begin{equation}\label{vanish4}
   	\frac{1}{2}e_1 \otimes e_2 \otimes e_2 \otimes e_1
   \end{equation}
   so that (\ref{vanish3})$+$(\ref{vanish4})$=0$.

The sign convention above is significant when we generalize this formula for $\det_4$ to a general formula for $\det_n$. For instance, let us consider the case $n=6$ and expect a formula satisfying the above rules.
We take the index set $I_6$ as
\begin{equation*}
\Scale[0.9]{
	I_6=\{(i_1,j_1,i_2,j_2,i_3,j_3) \in \mathbb{Z}^6~|~i_1,j_1,i_2,j_2,i_3,j_3 \in [6] \text{ are pairwise different, and } i_{p}<j_{p} \text{ for each } 1 \leq p \leq 3\}.
	}
\end{equation*}
Let 
\begin{equation*}
	S_{i_1,j_1,i_2,j_2,i_3,j_3}=sgn\begin{pmatrix}
		1 & 2 & 3 & 4 & 5 & 6\\
		i_1 & j_1 & i_2 & j_2 & i_3 & j_3
	\end{pmatrix}.
\end{equation*}
Our observation leads to a formula
\begin{equation}\label{newdet6}
\Scale[0.9]{
		\begin{aligned}
			&\frac{1}{2}\sum_{I_{6}}S_{i_1,j_1,i_2,j_2,i_3,j_3}(e_{i_1}-e_{j_1})\otimes (e_{i_2}-e_{j_2})\otimes (e_{i_3}-e_{j_3})\otimes (e_{i_3}+e_{j_3}) \otimes  (e_{i_2}+e_{j_2})\otimes (e_{i_1}+e_{j_1})\\
		&+\frac{1}{2}\sum_{I_{6}}(-1)^3S_{i_1,j_1,i_2,j_2,i_3,j_3}(e_{i_1}+e_{j_1})\otimes (e_{i_2}+e_{j_2})\otimes (e_{i_3}+e_{j_3})\otimes (e_{i_3}-e_{j_3}) \otimes  (e_{i_2}-e_{j_2})\otimes (e_{i_1}-e_{j_1})
		\end{aligned}
		}
	\end{equation}
consisted of $180$ decomposable tensors which coincides with $det_6$. Note that the second line of (\ref{newdet6}), we need to multiply $(-1)^3 = (-1)^{6/2}$ on each of the term. Let us have a brief look why the sign $(-1)^3$ is necessary. For a fixed $(i_1,j_1,i_2,j_2,i_3,j_3) \in I_6$, we take the following term determined by this (multi-)index
\begin{equation*}
	\frac{1}{2}S_{i_1,j_1,i_2,j_2,i_3,j_3}(e_{i_1}-e_{j_1})\otimes (e_{i_2}-e_{j_2})\otimes (e_{i_3}-e_{j_3})\otimes (e_{i_3}+e_{j_3}) \otimes  (e_{i_2}+e_{j_2})\otimes (e_{i_1}+e_{j_1})
\end{equation*}
from (\ref{newdet6}). We expand this term and focus on the summand
\begin{equation}\label{Exdet6_1}
	\frac{1}{2}S_{i_1,j_1,i_2,j_2,i_3,j_3}e_{i_1} \otimes (-e_{j_2}) \otimes e_{i_3} \otimes e_{j_3} \otimes e_{i_2} \otimes e_{j_1}
\end{equation}
which will contribute to $\det_6$. The unique candidate which contains this summand is
\begin{equation*}
	\frac{1}{2}(-1)^3S_{i_1,j_1,i_2,j_2,i_3,j_3}(e_{i_1}+e_{j_1})\otimes (e_{i_2}+e_{j_2})\otimes (e_{i_3}+e_{j_3})\otimes (e_{i_3}-e_{j_3}) \otimes  (e_{i_2}-e_{j_2})\otimes (e_{i_1}-e_{j_1})
\end{equation*}
that appears on the second line of (\ref{newdet6}). We extract the summand
\begin{equation}\label{Exdet6_2}
	\frac{1}{2}S_{i_1,j_1,i_2,j_2,i_3,j_3}(-1)^3e_{i_1} \otimes e_{j_2} \otimes e_{i_3} \otimes (-e_{j_3}) \otimes e_{i_2} \otimes (-e_{j_1})
\end{equation}
which is the same summand as (\ref{Exdet6_1}). It is easy to verify that the sign $(-1)^3$ at (\ref{Exdet6_2}) changes the components $e_{j_2}$, $-e_{j_3}$ and $-e_{j_1}$ to $-e_{j_2}$, $e_{j_3}$ and $e_{j_1}$, respectively. Therefore, we have
\begin{equation*}
	(\ref{Exdet6_1})+(\ref{Exdet6_2})=-S_{i_1,j_1,i_2,j_2,i_3,j_3}e_{i_1} \otimes e_{j_2} \otimes e_{i_3} \otimes e_{j_3} \otimes e_{i_2} \otimes e_{j_1}=S_{i_1,j_2,i_3,j_3,i_2,j_1}e_{i_1} \otimes e_{j_2} \otimes e_{i_3} \otimes e_{j_3} \otimes e_{i_2} \otimes e_{j_1}
\end{equation*}
which is compatible with the summand in the Leibniz formula (\ref{Detn}).
The sign $(-1)^k$ in the second line of (\ref{GeneralFormula}) helps us to save the terms, and also assures a number of vanishing of the bad terms which do not contribute to $\det_n$. Motivated by these symmetries and sign conventions, we generalize (\ref{newdet4}), (\ref{newdet6}) and obtain a formula for the arbitrary $2k \times 2k$ determinant tensor $det_{2k}$ for each $k>0$ as in Theorem \ref{MainThm}. 

\begin{proof}[Proof of Theorem \ref{MainThm}]\
	We are going to directly show by definition. This can be done by showing the following two conditions: when we expand our formula (\ref{GeneralFormula}), then
	\begin{itemize}
		\item [(\romannumeral1)] each term contributes to (\ref{Detn}) survives with the coefficient $1$ and the correct sign, and
		\item [(\romannumeral2)] each term which does not belong to (\ref{Detn}) vanishes.
	\end{itemize}
As a result, we will see that our formula (\ref{GeneralFormula}) is identical to $\det_n$ defined via the Leibniz formula (\ref{Detn}).

	\begin{itemize}
		\item [(\romannumeral1)] Consider an arbitrary term 
			\begin{equation}\label{Giventerm}
				S_{a_1,a_2,...,a_{2k}}e_{a_1} \otimes e_{a_2} \otimes \cdots \otimes e_{a_{2k-1}} \otimes e_{a_{2k}}
			\end{equation}
			appears in (\ref{Detn}). Among the terms in the formula (\ref{GeneralFormula}), only two terms
			\begin{equation}\label{Proof1Cand1}
				\frac{1}{2}S_{a_1,a_{2k},a_2,a_{2k-1},...,a_{k},a_{k+1}}(e_{a_1}-e_{a_{2k}}) \otimes \cdots \otimes (e_{a_{k}}-e_{a_{k+1}}) \otimes (e_{a_{k}}+e_{a_{k+1}}) \otimes \cdots  \otimes (e_{a_1}+e_{a_{2k}})
			\end{equation}
			and
			\begin{equation}\label{Proof1Cand2}
				\frac{1}{2}S_{a_1,a_{2k},a_2,a_{2k-1},...,a_{k},a_{k+1}}(-1)^k(e_{a_1}+e_{a_{2k}})  \otimes \cdots \otimes (e_{a_{k}}+e_{a_{k+1}}) \otimes (e_{a_{k}}-e_{a_{k+1}}) \otimes \cdots  \otimes (e_{a_1}-e_{a_{2k}})
			\end{equation}
			contain the term (\ref{Giventerm}). Note that the above expression seems to make sense only when $a_1 < a_{2k}$, $a_{2} < a_{2k-1}$ and so on, however, if it is not the case so that $a_{i}>a_{2k+1-i}$, we switch their positions and rewrite the component $e_{a_{i}} - e_{a_{2k+1-i}}$ as $-(e_{a_{2k+1-i}} - e_{a_i})$  in order to make the index $(a_1, a_2, \cdots, a_{2k}) \in I_{2k}$. This is reasonable since $S_{a_1, \cdots, a_{2k+1-i}, \cdots, a_i, \cdots, a_{2k}} = - S_{a_1, \cdots, a_{i}, \cdots, a_{2k+1-i}, \cdots, a_{2k}}$ the sign changes for this transposition. 
			 In particular, we are always able to take both (\ref{Proof1Cand1}) and (\ref{Proof1Cand2}) appear as summands in (\ref{GeneralFormula}). It is easy to check that both (\ref{Proof1Cand1}) and (\ref{Proof1Cand2}) yield the summand
			\begin{equation*}
				\frac{1}{2}S_{a_1,a_2,...,a_{2k}}e_{a_1} \otimes e_{a_2} \otimes \cdots \otimes e_{a_{2k-1}} \otimes e_{a_{2k}}.
			\end{equation*}
			when we expand them. Their sum is exactly the one we wanted, and no other terms in the (\ref{GeneralFormula}) can contribute to this summand.
						
		\item [(\romannumeral2)] For each index $(i_1,j_1,i_2,j_2....,i_k,j_k) \in I_{2k}$, we couple the two terms
			\begin{equation}\label{detminus}
				\frac{1}{2}S_{i_1,j_1,...,i_k,j_k}(e_{i_1}-e_{j_1})\otimes (e_{i_2}-e_{j_2})\otimes \cdots \otimes (e_{i_k}-e_{j_k})\otimes (e_{i_k}+e_{j_k}) \otimes \cdots \otimes (e_{i_2}+e_{j_2})\otimes (e_{i_1}+e_{j_1}).
			\end{equation}
			and 
			\begin{equation}\label{detplus}
				\frac{1}{2}(-1)^kS_{i_1,j_1,...,i_k,j_k}(e_{i_1}+e_{j_1})\otimes (e_{i_2}+e_{j_2})\otimes \cdots \otimes (e_{i_k}+e_{j_k})\otimes (e_{i_k}-e_{j_k}) \otimes \cdots \otimes (e_{i_2}-e_{j_2})\otimes (e_{i_1}-e_{j_1}).
			\end{equation}
			For a fixed index $(i_1,j_1,i_2,j_2....,i_k,j_k) \in I_{2k}$, there are plenty of bad terms which do not appear in $\det_n$ when we expand (\ref{detminus}). For example, the term $\frac{1}{2}S_{i_1,j_1,i_2,j_2,i_3,j_3}e_{i_1} \otimes e_{i_2} \otimes e_{i_3} \otimes e_{j_3} \otimes e_{j_2} \otimes e_{i_1}$ is not in the formula (\ref{Detn}) since both the first and the last components are the same as $e_{i_1}$.  Indeed, we see that ``bad summands'' have at least one pair of coinciding components. We need to observe that on which term may yield such a bad summand, and which term may kill such a bad summand arises from expanding another term. From our construction of (\ref{GeneralFormula}), it is clear that a bad term cannot contain a triple or bigger tuple of coinciding components -- only a number of pairs can happen.
			
			\begin{itemize}
			\item [Case 1)] Assume that the number of pairs of duplicated components is odd. \\
				: Let $T_1$ be such a summand appears in the expansion of the form (\ref{detminus}), and let $p$ denote the number of pairs of duplicated components in $T_1$. We will show that $-T_1$ occurs from the expansion of (\ref{detplus}) and kill $T_1$. If there is a component of $T_1$ which is used twice, then it will appear at the $t$-th and the $(2k+1-t)$-th component for some $t \in [k]$. The number of components of $T_1$ which are used only once is $2k-2p$ (possibly it can be $0$). Hence, when we expand (\ref{detplus}), we see that there is a summand of the form 
				\begin{equation*}
					(-1)^k(-1)^{k-p} T_1=(-1)^{2k-p} T_1 = -T_1.
				\end{equation*}
				
				since $p$ is odd. Possibly there could be another index $(i_1^{\prime}, j_1^{\prime}, \cdots, i_{k}^{\prime}, j_k^{\prime}) \in I_{2k}$ which yield the same bad summand $T_1$ in the expansion of (\ref{detminus}) corresponding to this new index, however, even in the case it is killed by the term of the form (\ref{detplus}) corresponding to the same index.

			\item [Case 2)] Assume that the number of pairs of duplicated components is even.\\
				: Let $T_2$ be a such summand from (\ref{detminus}), and let $q$ denote the number of pairs of duplicated components in $T_2$. To kill this $T_2$, we need another term than (\ref{detplus}) since the expansion of (\ref{detplus}) corresponding to the same index also contains $T_2$, not $-T_2$.
				
				Let $J$ be the set of the components $e_j$'s which are not used at $T_2$. Clearly $|J|=q >0$ is positive and even, so we may choose two elements $e_r \neq e_l \in J$. The original term $\mathcal T$ determined by (\ref{detminus}) corresponding to the given index cannot have $(e_r-e_l)$ as a component, i.e., it is not of the form $\frac{1}{2} C \otimes (e_r-e_l) \otimes C'$ for any tensors $C$ and $C'$. Then the indices $r, l$ determine the indices $z, w \in [2k]$ so that $\mathcal T$ have both $e_r-e_z$ and $e_l-e_w$ as its components (we switch $r, z$ and/or $l, w$ if necessary), where $r, l, z, w$ are pairwise different. Indeed,
				\begin{equation*}
					\mathcal T = \frac{1}{2}S_{...,r,z,...,l,w,...} C_1 \otimes (e_r-e_z) \otimes C_2 \otimes (e_l-e_w) \otimes C_3
				\end{equation*}
				for some tensors $C_1,C_2$ and $C_3$. 
				It is clear that $e_z$ and $e_w$ are duplicated components in $T_2$. We take another term in (\ref{GeneralFormula}), namely,
				\begin{equation}\label{Proof2took}
					\frac{1}{2}S_{...,l,z,...,r,w,...} C_1 \otimes (e_l-e_z) \otimes C_2 \otimes (e_r-e_w) \otimes C_3.
				\end{equation}
				If we expand (\ref{Proof2took}), then we find $-T_2$ since the sign
				\begin{equation*}
					S_{...,l,z...,r,w,...}=-S_{...,r,z,...,l,w,...}
				\end{equation*}
				is changed. In particular, we find a term, determined by another index, which involves $-T_2$ which eliminates $T_2$ arose from $\mathcal T$. Similarly, a bad term $T_2$ arises from $\mathcal T^{\prime}$ of the form (\ref{detplus}) is also eliminated by another term.
			
		For instance, in the case of $det_6$, a summand
				\begin{equation}\label{Proof2Ex2_given}
					T_2:=\frac{1}{2}S_{1,6,2,5,3,4} ~e_1 \otimes e_2 \otimes e_3 \otimes e_4 \otimes e_2 \otimes e_1
				\end{equation} 
				appears in the expansion of the term 
				\begin{equation}\label{Proof2Ex2_1}
					\frac{1}{2}S_{1,6,2,5,3,4}(e_1-e_6)\otimes (e_2-e_5) \otimes (e_3-e_4)\otimes (e_3+e_4) \otimes (e_2+e_5)\otimes (e_1+e_6).
				\end{equation}
				Here, exactly two components $e_1$ and $e_2$ are used twice in $T_2$. The set of components which are not used in $T_2$ is $J:=\{e_5,e_6\}$. We take the following term from (\ref{GeneralFormula})
				\begin{equation}\label{Proof2Ex2_2}
					\frac{1}{2}S_{1,5,2,6,3,4}(e_1-e_5)\otimes (e_2-e_6) \otimes (e_3-e_4)\otimes (e_3+e_4) \otimes (e_2+e_6)\otimes (e_1+e_5)
				\end{equation}
				so that the indices $5$ and $6$ are interchanged. We have the summand $-T_2$ in the expansion of (\ref{Proof2Ex2_2})
				\begin{equation}\label{ProofEx2_took}
					\begin{aligned}
						&\frac{1}{2}S_{1,5,2,6,3,4}~e_1 \otimes e_2 \otimes e_3 \otimes e_4 \otimes e_2 \otimes e_1\\
						&=-\frac{1}{2}S_{1,6,2,5,3,4}~e_1 \otimes e_2 \otimes e_3 \otimes e_4 \otimes e_2 \otimes e_1\\
						&=-T_2,
					\end{aligned}
				\end{equation}
				as desired.
			\end{itemize}
	\end{itemize}
\end{proof}

\begin{cor}
\begin{equation*}
	\operatorname{Crank}(det_n) \leq \operatorname{rank}(det_n) \leq \frac{n!}{2^{\lfloor(n-2)/2 \rfloor}}~ \text{ if } ~char(\mathbb{K}) \neq 2
\end{equation*}
\end{cor}
\begin{proof}
	When $n$ is even, it comes from the formula (\ref{GeneralFormula}) which is the sum of
		\begin{equation*}
		2 \times {n \choose 2}\times {n-2 \choose 2} \times \cdots \times {2 \choose 2}=\frac{n!}{2^{(n-2)/2}}=\frac{n!}{2^{\lfloor(n-2)/2 \rfloor}}
	\end{equation*}
	decomposable tensors. 
	
	When $n$ is odd, note that $det_{n-1}$ is the sum of
	\begin{equation*}
		\frac{(n-1)!}{2^{(n-3)/2}}
	\end{equation*}
	decomposable tensors. Thanks to the Laplace expansion formula, we can represent $det_n$ as a sum of
	\begin{equation*}
		\frac{(n-1)!}{2^{(n-3)/2}} \times n=\frac{n!}{2^{(n-3)/2}}=\frac{n!}{2^{\lfloor(n-2)/2 \rfloor}}
	\end{equation*}
	decomposable tensors as well.
\end{proof}

Very recently, Houston, Goucher, and Johnston reported another explicit formula of $\det_n$\cite{houston2023new} which implies
\begin{equation*}
	\operatorname{rank}(det_n) \leq B_n
\end{equation*}
where $B_n$ is the $n$-th Bell number. We compare this result and our formula: let $C_n = \frac{n!}{2^{\lfloor(n-2)/2 \rfloor}}$ be an upper bound we found

\begin{table}[h!]
\centering
\begin{tabular}{ |c|c|c|c|c|c|c|c|c| } 
\hline 
 $n$ & 2 & 3 & 4 & 5 & 6 & 7 & 8 & $\cdots$ \\ 
 \hline 
 $B_n$ & 2 & 5 & 15 & 52 & 203 & 877 & 4140 & $\cdots$\\ 
 \hline
 $C_n$ & 2 & 6 & \textbf{12} & 60 & \textbf{180} & 1260 & 5040 & $\cdots$\\ 
 \hline
\end{tabular}
\caption{When $n=4,6$, we have $C_n \leq B_n$.}
\label{table1}
\end{table}

We observe that $C_n \leq B_n$ only when $n=4, 6$, however, we expect that our formula (\ref{GeneralFormula}) has a potential to be improved since there are a lot of symmetries inside it. For instance, our formula implies an alternative proof for $\operatorname{Wrank} (\det_4) \le 96$ (see also 
\cite{johns2022improved}), so it might lead to a better upper bound. We also expect that our simple formula gives an effective way to compute the determinant of matrices over various rings and maps on exterior algebras.

\bibliography{Det}
\vspace{0.5cm}

\end{document}